\documentclass[reqno]{amsart}

\usepackage{amssymb,latexsym}
\usepackage{amsfonts}
\usepackage{dsfont}
\usepackage[pdftex]{hyperref}
\usepackage{color}
\usepackage{graphicx}
\usepackage{parskip}
\usepackage{mathrsfs}

\theoremstyle{plain}

\newtheorem{theorem}{Theorem}[section]
\newtheorem{lemma}{Lemma}[section]

\numberwithin{equation}{section}

\newcommand{\p}{\partial}
\newcommand{\m}{\mathbb}

\newcommand{\rr}{\mathbb{R}}

\newcommand{\ci}{\mathbb{T}}

\newcommand{\wh}{\widehat}
\newcommand{\ee}{\varepsilon}

\begin{document}

\title[Classical Solutions of the FW Equation]{Classical Solutions of the Fornberg-Whitham Equation}
\author{Georgia Burkhalter, Ryan C. Thompson$^*$, \& Madison Waldrep}
\date{December 19, 2022. $^*$Corresponding author: ryan.thompson@ung.edu}

\keywords{Fornberg-Whitham equation, Cauchy problem, Sobolev spaces, well-posedness, classical solutions, diffeomorphisms, conserved quantities.}

\subjclass[2020]{Primary: 35Q53}

\address{Department of Mathematics\\University of North Georgia\\ Dahlonega, GA 30597}
\email{gaburk1185@ung.edu}
\address{Department of Mathematics\\University of North Georgia\\ Dahlonega, GA 30597}
\email{ryan.thompson@ung.edu}
\address{Department of Mathematics\\University of North Georgia\\ Dahlonega, GA 30597}
\email{mbwald7063@ung.edu}

\begin{abstract}
In this paper, we prove well-posedness in $C^1(\rr)$ (a.k.a. classical solutions) of the Fornberg-Whitham equation.  To achieve this objective, we study its weak formulation under a Lagrangian framework.  Applying the fundamental theorem of ordinary differential equations to the generated semi-linear system, we then construct a unique solution to the equation that is continuously dependent on the initial data.  These results improve upon others in Sobolev and Besov spaces.
\end{abstract}

\maketitle

%%%%%%%%%%%%%%%%%%%%%%%%%%%%%%
%
%
%		Introduction
%
%
%%%%%%%%%%%%%%%%%%%%%%%%%%%%%%
\section{Introduction}
We consider the Cauchy problem for the Fornberg-Whitham (FW) equation
\begin{equation}
\label{fw}
\begin{cases}
u_{xxt}-u_t+\frac92u_xu_{xx}+\frac 32uu_{xxx}-\frac32uu_x+u_x=0 \\ 
u(x,0)=u_0(x),
\end{cases}
\end{equation}
when $x \in \rr$ or $\ci$ and $ t \in \m{R}$.  This equation was first written down in 1967 by Whitham \cite{whitham} and again by Whitham and Fornberg \cite{fw} as a model for breaking waves.  Additionally, \eqref{fw} also admits the following nonlocal form
\begin{equation}
\label{fwnl}
u_t+\frac32uu_x=\p_x(1-\p_x^2)^{-1}u,
\end{equation}
where $\mathcal{F}[ {(1-\p_x^2)^{-1}f }]= \frac1{1+\xi^2} \wh f(\xi)$ for any test function $f$.  

We show that the FW equation is well-posed in the space of bounded and continuously differentiable functions on the real line, denoted $C^1$, and equipped with the norm
\[
\|f\|_{C^1} = \sup_{x\in \rr}|f(x)|+\sup_{x\in \rr}\left|\frac{d}{dx}f(x)\right|.
\]
More precisely, if we endow \eqref{fwnl} with initial data $u_0 \in C^1$, we have a corresponding uniquely constructed solution $u(x,t) \in C([-T,T];C^1)$ and the solution is continuously dependent on the initial data.  Furthermore, we find a lifespan estimate that depends on the size of the initial data.  In fact, we find that during this lifespan, the solution remains bounded by two times the size of the initial data.

If we write the FW equation in this form, we see that it resembles a special case of a family of nonlinear wave equations
\begin{equation}
\label{burgers}
u_t+\alpha uu_x=\mathscr{L}(u),
\end{equation}
where $\alpha \in \m{R}$ and $\mathscr{L}$ is a linear operator with constant coefficients, which has been studied by multiple authors.  In both \cite{whitham} and \cite{fw}, the FW equation \eqref{fwnl} was compared with the Korteweg-de Vries (KdV) equation
\[
u_t+6uu_x=-\p_x^3u,
\]
which was first derived by Boussinesq \cite{b} in 1877 and then by Korteweg and de Vries \cite{kdv} in 1895.  Indeed, we may see this comparison by letting $\alpha=6$ and $\mathscr{L} = -\p_x^3$.

It is well-known that the KdV equation has a bi-Hamiltonian structure and is completely integrable \cite{dj}.  Additionally, the KdV equation admits the soliton solution
\[
u(x,t) = \frac{c}{2}\text{sech}^2\left[\frac{\sqrt{c}}{2}(x-ct)\right],
\]
which maintains a constant shape and moves at a constant velocity.  KdV, however, does not exhibit the property of breaking waves and hence mathematicians and physicists alike continued their search for such an equation.

It was then in 1967, while Whitham \cite{whitham} was exploring applications to water waves he wrote down the FW equation \eqref{fwnl} and noted that it produced the so-called ``peakon" solution with maximum amplitude of 8/9, i.e. a soliton that is not differentiable at its peak.  Then in 1978, Whitham and Fornberg \cite{fw} researched both numerical and theoretical results on the FW equation and wrote down the explicit peakon solution
\[
u(x,t)=\frac89e^{-\frac12\left|x-\frac43t\right|},
\]
along with noting the wave breaking properties.  In fact, $8/9$ was found as a limiting behavior of the amplitude of the above exponential peakon.  Furthermore, the FW equation has been shown to exhibit the following conservation laws
\[
E_1(u) = \int_\rr udx, \ \ \ E_2(u) = \int_\rr u^2dx, \ \ \ E_3(u) = \int_\rr(u(1-\p_x^2)^{-1}u-u^3)dx.
\]

In regards to the discovery of wave breaking, Whitham and Fornberg's mathematical arguments were incredibly formal.  We refer the reader to their original paper \cite{fw} and to Contantin and Escher \cite{ce} for further discussion.  Since then, many authors have researched and found other nonlinear wave equations similar to FW.

Indeed, when $\alpha = 1$ and $\mathscr{L}(u) = -\p_x(1-\p_x^2)^{-1}\left(u^2+\frac12u_x^2\right)$, we have that \eqref{burgers} becomes the celebrated Camassa-Holm (CH) equation
\[
u_t+uu_x = -\p_x(1-\p_x^2)^{-1}\left(u^2+\frac12u_x^2\right),
\]
which models the unidirectional propagation of shallow water waves over a flat bottom \cite{ch, ch1} as well as axially symmetric waves in hyperelastic rods \cite{dai}.  The CH equation also admits peakon solutions \cite{con1, ce1, ce2}  of the form
\[
u(x,t) = ce^{-|x-ct|}, \ \ c>0.
\]
Furthermore, solutions exhibit blow-up in the form of wave breaking \cite{ce}.  Global conservative solutions were also established by Bressan and Constantin \cite{bc}, where they transform the CH equation into a semilinear system of ODEs and obtain solutions as fixed points of a contractive transformation.  Furthermore, classical solutions in the periodic setting were also investigated for the CH equation in Misio\l ek \cite{mis} where the author utilizes the structure of the equation along with delicate commutator estimates and methods from Arnold's paper \cite{arn}.  An extension of classical solutions to the real line was recently given by Holmes and Thompson\cite{ht2}.  It's in these above results where we draw inspiration for our paper.

The CH equation possesses many other 
remarkable properties such as infinitely many conserved quantities, a bi-Hamiltonian structure and a Lax pair.
For more information about how CH arises in the context of hereditary symmetries we refer to \cite{ff}.  Concerning it's physical relevance, we refer the reader to the works by Johnson \cite{j1}, \cite{j2} and Constantine and Lannes \cite{cl}.

For $\alpha = 1$ and $\mathscr{L}(u) = -\frac32\p_x(1-\p_x^2)^{-1}(u^2)$, \eqref{burgers} becomes the Degasperis - Procesi (DP) equation
\[
u_t+uu_x = -\frac32\p_x(1-\p_x^2)^{-1}(u^2)
\]
which was discovered by Degasperis and Procesi \cite{dp} when they were in search of asymptotically integrable PDEs.  Degasperis, Holm, and Hone proved the equation was integrable \cite{dhh} by constructing a Lax pair, and showed the DP equation admits peakon solutions of the form
\[
u(x,t) = \pm ce^{-|x-ct|}, \ \ c>0.
\]
It was also shown in \cite{ely, ly1, ly2} that blow-up occurs in finite time.

It's important to mention that the aforementioned equations are integrable; they possess infinitely many conserved quantities, an infinite hierarchy of quasi-local symmetries, a Lax pair and bi-Hamiltonian structure.  In contrast, the FW equation is not integrable.  However, it has been shown that FW is locally well-posed in both Sobolev and Besov spaces $H^s$ and $B_{p,r}^s$ for $s>3/2$ and ill-posed in Besov spaces $B_{2,\infty}^{3/2}$ and $B_{p,\infty}^s$ for $s>1+1/p$ and $p \in [1,\infty]$.  Furthermore, while FW is locally well-posed in particular Sobolev and Besov spaces, the regularity of the data-to-solution map is sharp.  Indeed, it has been shown that the data-to-solution map is continuous but not uniformly continuous.  For more on well-posedness results, we refer the reader to Holmes et. al. \cite{h2, ht1} and Guo \cite{guo}.

To our knowledge, beyond the results established in \cite{ht2, mis}, little has been done to show that other wave equations exhibit the property of yielding classical solutions.  Our goal is to establish well-posedness of FW in $C^1(\rr)$ which, along with the tools found in \cite{ht2}, should provide an apparatus for others seeking classical solutions to other nonlinear dispersive equations.  We now state our main results.

\begin{theorem}
\label{wp}
The initial value problem for the Fornberg-Whitham equation is well posed in $C^1$.
\end{theorem}
In regards to the aforementioned theorem, we show how to construct a solution $u \in C([-T,T]; C^1) \cap C^1([-T,T];C)$ from the solution to a system of ODEs. Since it is constructed from a system of ODEs, and the ODE system has a unique solution, it follows that there is a uniquely constructed solution. Using the ODEs, we can also show the solution depends continuously on the initial data, and additionally we are able to estimate a minimum lifespan  of the solution  from the Lipshitz constant
\[
T  =  \frac{9}{100\|u_0\|_{C^1}}.
\] 
To construct the solution, we will also need to define a diffeomorphism from the ODE's. We will find that a sufficient condition to guarantee that the aforementioned diffeomorphism is invertible by restricting $|t|<T$.

Combining these estimates, we will find that the solution satisfies the size estimate
$$
\sup_{t\in [-T,T]}  \|u(t)\|_{C^1} \le 2 \|u_0\|_{C^1}.
$$
 We shall also show the following estimate on the data--to--solution map.
\begin{theorem} 
\label{holder}
If $u_0\in C^1$, then the data-to-solution map is H\"older continuous from $C^\alpha$ to $C([-T,T];C^\alpha)$, where $0\le\alpha <1$. 
\end{theorem}

The proof of Theorem \ref{wp} will be organized as follows.  First we write the FW equation down as a system of ordinary differential equations (ODEs), which we can solve via the ODE theorem. Then, we show that we can construct a function $u \in C([-T,T]; C^1)$ such that $u$ is the unique solution to the FW equation.  

We let $C^k( \rr)$, $k = 0,1,2,\dots$ be the Banach space of $k$ continuously differentiable functions, which are bounded, equiped with the norm 
$$
\|f\|_{C^k} = \sum_{n=0}^k \sup_{x\in \rr} \left|\frac{d^n}{dx^n} f(x) \right|.
$$
Furthermore, it will be understood that
\[
C^k \doteq C^k(\rr)
\]
throughout the paper since all calculations will be over the real line.
Throughout the paper we will let $\text{Diff}^k$ denote the set of $k$-times continuously differentiable diffeomorphisms of the line $\rr$. 
 This set is a topological group under composition of diffeomorphisms.

In the remainder of this document, we will also use subscripts to denote derivatives applied before composition; i.e. the notation $ u_x(\eta(x,t),t) = (\p_x u)(\eta(x,t),t)$ and $\p_x u(\eta(x,t),t) = u_x(\eta(x,t),t) \p_x \eta(x,t)$. 

An outline of our paper is as follows.  In section two we construct our system of ODEs.  Then in section three we apply the ODE theorem to the aforementioned system.  In section four we construct our solution $u(x,t)$ and show that it solves the FW equation, which concludes the proof of Theorem \ref{wp}.  Finally, in section five we investigate H\"{o}lder continuity for the data-to-solution map.

%%%%%%%%%%%%%%%%%%%%%%%%%%%%
%
%
%			The Semi-linear System
%
%
%%%%%%%%%%%%%%%%%%%%%%%%%%%%
\section{A Semi-linear System of ODE's}

We begin by showing how one formally constructs an equivelent ODE system to the FW equation. Assuming a solution, $u$, exists and is a $C^\infty$ solution of the FW initial value problem, we have our trajectories satisfy the ODE
\begin{equation*}
\begin{cases}
\eta_t(x,t)=\frac32u(\eta(x,t),t) \\ 
\eta(x,0)=x.
\end{cases}
\end{equation*}

Moreover, the above ODE has a unique solution $\eta(x,t)$ which is also continuously differentiable,
therefore, we may define
\begin{align}
&w(x,t) = u(\eta(x,t),t),\quad
v(x,t) = u_x(\eta(x,t),t)  ,\quad q(x,t) = \eta_x(x,t),
\end{align}
and we see that we may easily obtain $u(x,t)$ from the composition 
$$
u = w \circ \eta^{-1}.
$$
We will first find a system of equations satisfied by $w$, $v$ and $q$, and then show that this system of equations is indeed an ODE system, and therefore the solutions are uniquely defined. Using $w$, we will then construct $\eta$ and $u$ similarly to the above formal definitions. 

\textbf{Equation for w(x,t):} Differentiating $w$ with respect to $t$ we get
\begin{align*}
\partial_tw(x,t) &= \partial_tu(\eta(x,t),t) \\ 
&= u_t(\eta(x,t),t) + u_x(\eta(x,t),t)\frac \p{\p t}\eta(x,t) \\ 
&= u_t(\eta(x,t),t) + u_x(\eta(x,t),t)\frac32u(\eta(x,t),t)
\end{align*}

From the non-local form of the FW equation, this yields
\begin{equation}
\label{nl1}
\partial_tw(x,t) = \partial_x(1-\partial^2_x)^{-1}(u)(\eta(x,t),t)
\end{equation}
To evaluate the above equation, we write the non-local form as an integral.  Indeed we have that
$(1-\partial^2_x)^{-1}f=\frac12e^{-|x|}*f$.  Therefore,
\begin{equation}
\label{nl2}
\partial_x(1-\partial^2_x)^{-1}(u) =\frac12 \int_{\mathbb{R}} \partial_xe^{-|x-y|}(u(y,t))dy.
\end{equation}
\label{nl3}
We split the exponential into two pieces to apply the derivative operator and obtain
\begin{equation}
\label{nl4}
\partial_x(1-\partial^2_x)^{-1}(u) =\frac12 \int_{x}^{\infty}e^{-|x-y|}(u(y,t))dy - \frac12\int_{-\infty}^{x}e^{-|x-y|}(u(y,t))dy
\end{equation}
Now we evaluate at $x = \eta(x,t)$ to find
\begin{align}
\label{nl5}
(\partial_x(1-\partial^2_x)^{-1}(u))(\eta(x,t),t) = &\frac12 \int_{\eta(x,t)}^{\infty}e^{-|\eta(x,t)-y|}(u(y,t))dy \nonumber \\ 
&- \frac12\int_{-\infty}^{\eta(x,t)}e^{-|\eta(x,t)-y|}(u(y,t))dy
\end{align}
We note that our integration is in respect to $y$. Applying a change of variables $y=\eta(z,t)$ yields
\begin{align}
\label{nl6}
(\partial_x(1-\partial^2_x)^{-1}(u))(\eta(x,t),t) =&\frac12 \int_{\eta(x,t)}^{\infty}e^{-|\eta(x,t)-\eta(z,t)|}(u(\eta(z,t),t))\eta_z(z,t)dz \nonumber
\\ &- \frac12\int_{-\infty}^{\eta(x,t)}e^{-|\eta(x,t)-\eta(z,t)|}(u(\eta(z,t),t))\eta_z(z,t)dz
\end{align}
Applying another change of variables $u(\eta(z,t),t) = w(z,t)$ and $q(z,t) = \eta_z(z,t)$ we obtain
\begin{align}
\label{nl7}
(\partial_x(1-\partial^2_x)^{-1}(u))(\eta(x,t),t) = &\frac12 \int_{x}^{\infty}e^{-|\eta(x,t)-\eta(z,t)|}w(z,t)q(z,t)dz \nonumber \\
&- \frac12\int_{-\infty}^{x}e^{-|\eta(x,t)-\eta(z,t)|}w(z,t)q(z,t)dz
\end{align}
Using the definition of $q(x,t)$ implies
\begin{equation}
\eta(z,t)-\eta(x,t) = \int_{x}^{z}q(y,t)dy,
\end{equation}
we find
\begin{align}
(\partial_x(1-\partial^2_x)^{-1}(u))(\eta(x,t),t) = &\frac12 \int_{x}^{\infty}e^{-|\int_{x}^{z}q(y,t)dy|}w(z,t)q(z,t)dz \nonumber \\ 
 &- \frac12\int_{-\infty}^{x}e^{-|\int_{x}^{z}q(y,t)dy|}w(z,t)q(z,t)dz.
\end{align}
For simplicity, we will redefine this equation as $P_1(w,q)$
\begin{equation}
P_1(w,q) = \frac12 \int_{x}^{\infty}e^{-|\int_{x}^{z}q(y,t)dy|}w(z,t)q(z,t)dz - \frac12\int_{-\infty}^{x}e^{-|\int_{x}^{z}q(y,t)dy|}w(z,t)q(z,t)dz.
\end{equation}
\textbf{Equation for q(x,t):} Next, we find $\partial_tq(x,t)$ satisfies
\begin{align}
\partial_tq(x,t) &= \partial_t\eta_x(x,t) \nonumber \\ 
&= \partial_x\eta_t(x,t) \nonumber \\ 
&=\frac32u_x(\eta(x,t),t)(\eta_x(x,t)).
\end{align}
So we obtain
\begin{equation}
\label{qeqn}
\partial_tq(x,t) = \frac32v(x,t)q(x,t)
\end{equation}
\textbf{Equation for v(x,t):} Finally, we determine an equation for $\partial_tv(x,t)$.  By the definition of $v$ we have
\begin{align}
\label{veqn}
\partial_tv(x,t) &= \partial_tu_x(\eta(x,t),t) \nonumber \\ 
&= u_{xt}(\eta(x,t),t)+u_{xx}(\eta(x,t),t)\eta_t(x,t) \nonumber  \\ 
& = u_{xt}(\eta(x,t),t)+u_{xx}(\eta(x,t),t)\frac32u(\eta(x,t),t)
\end{align}
By taking a spatial derivative of \eqref{fwnl}, we find
\[
u_{xt}+\frac32u_x^2+\frac32uu_{xx} = \p_x^2(1-\p_x)^{-1}(u).
\]
Noting that $\p_x^2(1-\p_x^2)^{-1}f = (1-\p_x^2)^{-1}f - f$, we have
\[
u_{xt}+\frac32u_x^2+\frac32uu_{xx} = (1-\p_x)^{-1}(u) - u.
\]
Plugging back into \eqref{veqn} yields
\begin{equation*}
\partial_tv(x,t) = \frac12\int_{\mathbb{R}}e^{-|\int_{x}^{z}q(y,t)dy|}w(z,t)q(z,t)dz-w-\frac32v^2.
\end{equation*}
Now let the integral expression be replaced with $P_2(w,q)$ and this yields
\begin{equation}
\label{veq}
\partial_tv(x,t) = P_2(w,q)-w - \frac32v^2 .
\end{equation}
Therefore, the initial value problem for the FW equation is formally equivalent to the following system of ordinary differential equations
\begin{equation}
\begin{cases}
\partial_tw = P_1(w,q), \\ \partial_tv = P_2(w,q)-w - \frac32v^2, \\ \partial_tq = \frac32vq
\end{cases}
\end{equation}
with initial data
\begin{equation}
\begin{cases}
w(x,0) = u_0(x), \\ v(x,0) = \partial_xu_0(x), \\ q(x,0)=1.
\end{cases}
\end{equation}
Next is to show that this system is an ODE in the space $C^1 \times C \times C$.

%%%%%%%%%%%%%%%%%%%%%%%%%
%
%
%			Application of the ODE Theorem
%
%
%%%%%%%%%%%%%%%%%%%%%%%%%
\section{Application of the ODE Theorem}
We now show that the aforementioned system of equations are ODE's in an appropriate Banach space. We define
\[
Y=\{(f,g,h) \in C^1\times C \times C : \|(f,g,h)\|_Y = \|f \|_{C^1} + \| g \|_C + \| h \|_C < \infty\}
\]
and we claim the above system is an ODE. To verify this claim, we rewrite the system of equations in the following way:
\begin{equation}
\label{ode}
\begin{cases}
     \frac{d}{dt}y = f(t,y)
     \\
     y(0) = y_0
\end{cases}
\end{equation}
where $y= (w,v,q), y(0) = \left(u_0(x), \frac{d}{dx}u_0(x), 1\right)$ and $f = \left(P_1(w,q), P_2(w,q)-w - \frac32v^2, \frac{3}{2}vq\right)$

We must show $y \rightarrow f(y)$ is Lipschitz continuous in a neighborhood of $y_0$.

\begin{lemma}
\label{lipex}
Let $B_{r_0} \subset Y$ be a small ball centered at $y_0 \in Y$ with a radius $r_0 < \frac{1}{9}$. Then the mapping $y \rightarrow f(t,y)$ is Lipschitz continuous from $B_{r_0}$ to $Y$; i.e. for any $y_1, y_2 \in B_{r_0} \subset Y$ with $\|y_i \|_Y < r \doteq r_0 + \| y_0 \|_Y, \ i = 1,2,$ there exists a constant c independent of r, such that
\[
\|f(y_1) - f(y_2) \|_Y \le \frac{50}{9}r\| y_1 - y_2 \|_Y.
\]
\end{lemma}

\begin{proof} We show, for $y_1= (w_1,v_1,q_1)$ and $y_2 = (w_2,v_2,q_2)$, in $B_{r_0}$, there exists a $L=L(r)$ such that
\[
 \|f(t, y_1) - f(t, y_2) \|_Y \le L \|y_1-y_2\|_Y = L \|w_1-w_2\|_{C^1} + L\|v_1-v_2\|_C + L\|q_1-q_2\|_C.
\]
Define $r_l = 1-r_0 < 1+r_0 = r_u$ (the upper and lower bounds of $q_1, q_2$). To show the above inequality, we split the estimate into three components, which are the three pieces of the $Y$ norm.
\begin{align}
\label{estimates}
\| f(t,y_1)-f(t,y_2)\|_Y &= \| P_1(w_1,q_1) -P_1(w_2,q_2) \|_{C^1} \nonumber \\ 
&+ \left\| P_2(w_1,q_1)-w_1 - \frac32v_1^2 -P_2(w_2,q_2)+w_2 +\frac32v_2^2 \right\|_C  \nonumber \\ 
&+ \left\| \frac{3}{2}v_1q_1 - \frac{3}{2}v_2q_2 \right\|_C \nonumber \\ 
&= E_1 + E_2 + E_3.
\end{align}

\subsection{Estimating $E_3$}
Here we start with the estimate for the third term since it's the easiest.
\[
\left\|\frac32v_1q_1-\frac32v_2q_2\right\|_C = \sup_{x\in\mathbb{R}}\frac32|v_1q_1-v_2q_2|
\]
We add and subtract $v_1q_2$ and apply the triangle inequality to obtain
\[
\left\|\frac32v_1q_1-\frac32v_2q_2\right\|_C \le  \sup_{x\in\mathbb{R}}\left(\frac32|v_1q_1 - v_1q_2| + \frac32|v_1q_2 - v_2q_2|\right)
\]
By assumption, $|v_1| \le r$, therefore
\[
\left\|\frac32v_1q_1-\frac32v_2q_2\right\|_C \le \sup_{x\in\mathbb{R}}\left(\frac32r|q_1-q_2| + \frac32|v_1q_2 - v_2q_2|\right)
\]
Using the functions that are bounded by $r$, ($q_1, q_2$ are bounded by $r_u$) we obtain
\[
\left\|\frac32v_1q_1-\frac32v_2q_2\right\|_C \le \sup_{x\in\mathbb{R}}\left(\frac32r|q_1-q_2| + \frac32r_u|v_1-v_2|\right)
\]
Hence,
\begin{equation}
\label{est-e3}
\left\|\frac32v_1q_1-\frac32v_2q_2\right\|_C \le \frac32r\left\|y_1-y_2\right\|_Y.
\end{equation}

\subsection{Estimating $E_2$} We proceed with the estimate for $E_2$ and shall do this in stages.  By triangle inequality we have that
\[
E_2 \le \|w_1-w_2\|_C + \left\|\frac32v_1^2-\frac32v_2^2\right\|_C +\|P_2(w_1,q_1) - P_2(w_2,q_2)\|_C
\]
The first two terms can be estimated in the same manner as we did for the estimate $E_3$.  Indeed, we have that
\[
\|w_1-w_2\|_C \le \|y_1-y_2\|_Y \ \ \text{and} \ \ \left\|\frac32v_1^2-\frac32v_2^2\right\|_C \le 3r\|y_1-y_2\|_Y.
\]
We now estimate $\|P_2(w_1,q_1) - P_2(w_2,q_2)\|_C$ which we will define as .  We have that
\[
\|P_2(w_1,q_1) - P_2(w_2,q_2)\|_C = \sup_{x\in\mathbb{R}} \left|\frac12\int_{\mathbb{R}} e^{-|\int_x^zq_1(y,t)dy|}w_1q_1-e^{-|\int_x^zq_2(y,t)dy|}w_2q_2dz\right|
\]
We add and subtract $e^{-|\int_x^zq_1(y,t)dy|}w_1q_2$, and then apply the triangle inequality to obtain
\begin{align}
\label{p2-est}
\|P_2(w_1,q_1) - P_2(w_2,q_2)\|_C \le &\sup_{x\in\mathbb{R}} \left|\int_{\mathbb{R}} e^{-|\int_x^zq_1(y,t)dy|}w_1q_1-e^{-|\int_x^zq_1(y,t)dy|}w_1q_2dz\right| \nonumber \\
&+  \sup_{x\in\mathbb{R}} \left|\int_{\mathbb{R}} e^{-|\int_x^zq_1(y,t)dy|}w_1q_2-e^{-|\int_x^zq_2(y,t)dy|}w_2q_2dz\right|
\end{align}
We take the supremum of the terms in the first integral to obtain
\begin{align*}
\|P_2(w_1,q_1) - P_2(w_2,q_2)\|_C &\le  \sup_{x\in\mathbb{R}}r \left|\int_{\mathbb{R}} e^{-r_l|x-z|}q_1-e^{-r_l|x-z|}q_2dz\right| \\
&+ \sup_{x\in\mathbb{R}}\left|\int_{\mathbb{R}} e^{-|\int_x^zq_1(y,t)dy|}w_1q_2-e^{-|\int_x^zq_2(y,t)dy|}w_2q_2dz\right| \\ 
 &=  \sup_{x\in\mathbb{R}}r\left|e^{-r_l|x|}*(q_1-q_2)\right| \\
&+ \sup_{x\in\mathbb{R}}\left|\int_{\mathbb{R}} e^{-|\int_x^zq_1(y,t)dy|}w_1q_2-e^{-|\int_x^zq_2(y,t)dy|}w_2q_2dz\right|
\end{align*}
where we used $r_l = 1-r_0 < |q| < 1+r_0 = r_u.$ By Young's inequality, we have
\[
\|P_2(w_1,q_1) - P_2(w_2,q_2)\|_C \le \sup_{x\in\mathbb{R}}\frac{2r}{r_l}|q_1-q_2| + \sup_{x\in\mathbb{R}}\left|\int_{\mathbb{R}} e^{-|\int_x^zq_1(y,t)dy|}w_1q_2-e^{-|\int_x^zq_2(y,t)dy|}w_2q_2dz\right|.
\]
Now for the second term we add and subtract $e^{-|\int_x^zq_1(y,t)dy|}w_2q_2$ and apply triangle inequality to obtain
\begin{align*}
\|P_2(w_1,q_1) - P_2(w_2,q_2)\|_C &\le \sup_{x\in\mathbb{R}}\frac{2r}{r_l}|q_1-q_2| + \sup_{x\in\mathbb{R}}\left|\int_{\mathbb{R}} e^{-|\int_x^zq_1(y,t)dy|}w_1q_2-e^{-|\int_x^zq_1(y,t)dy|}w_2q_2dz\right| \\
&+\sup_{x\in\mathbb{R}}\left|\int_{\mathbb{R}} e^{-|\int_x^zq_1(y,t)dy|}w_2q_2-e^{-|\int_x^zq_2(y,t)dy|}w_2q_2dz\right|
\end{align*}
For the first integral we take the supremum over terms and apply Young's inequality to find
\begin{align*}
\sup_{x\in\mathbb{R}}\left|\int_{\mathbb{R}} e^{-|\int_x^zq_1(y,t)dy|}w_1q_2-e^{-|\int_x^zq_1(y,t)dy|}w_2q_2dz\right| &\leq  \sup_{x\in\mathbb{R}}r_u\left|e^{-r_l|x|}*(w_1-w_2)\right| \\ 
& \leq \sup_{x\in\mathbb{R}}\frac{2r_u}{r_l}|w_1-w_2|.
\end{align*}
For the second integral we also take the supremum over terms and obtain
\[
\sup_{x\in\mathbb{R}}\left|\int_{\mathbb{R}} e^{-|\int_x^zq_1(y,t)dy|}w_2q_2-e^{-|\int_x^zq_2(y,t)dy|}w_2q_2dz\right| \leq rr_u\sup_{x\in\mathbb{R}}\left|\int_{\mathbb{R}} e^{-|\int_x^zq_1(y,t)dy|}-e^{-|\int_x^zq_2(y,t)dy|}dz\right|.
\]
We now estimate the remaining integral. 
By the estimate, $e^{-|x|} - e^{-|y|} =
e^{-|y|}(e^{-|x|+|y| }-1) \le e^{-|y|}(e^{|x-y| }-1)$ we have 
\begin{align*}
J = J(x,z) & =
e^{-|\int_x^z q_1(y,t) dy|}-e^{-|\int_x^z q_2(y,t) dy|} \le
e^{-|\int_x^z q_2(y,t) dy|} ( e^{|\int_x^z (q_1-q_2)(y,t) dy|}-1).
\end{align*}
Using the expansion $e^z - 1 = \sum_{n=1}^\infty \frac{z^n}{n!}$ and Minkowski's inequality to push the absolute values inside the integration, we obtain
\[
J 
 \le 
e^{-|\int_x^z q_2(y,t) dy|} \int_x^z |(q_1-q_2)(y,t)| dy \sum_{n=1}^\infty \frac{(\int_x^z |(q_1-q_2)(y,t) |dy)^{n-1}}{n!}.
\]
Next, we take the supremum over  $x$ of   $q_1 - q_2$, and  we have 
\[
J 
 \le 
e^{-|\int_x^z q_2(y,t) dy|}\|q_1 - q_2\|_{C} ||x-z|\sum_{n=1}^\infty \frac{(\int_x^z |(q_1-q_2)(y,t) |dy)^{n-1}}{n!}.
\]
Now we apply the boundedness of $q_j$ to obtain 
\begin{align*}
J 
& \le 
e^{-r_\ell |x-z|}\|q_1 - q_2\|_{C} |x-z|\sum_{n=1}^\infty \frac{((r_u - r_\ell) |x-z|)^{n-1}}{n!}
 \end{align*} 
Using $|x-z|  < e^{\frac12r_\ell |x-z|}  $  and $r_u - r_\ell = (1+r_0) - (1-r_0) = 2r_0$ and the assumption $r_0< \frac18 r_\ell $ (radius of the ball $\mathcal B_{r_0}$), we obtain 
\begin{align*}
J 
& \le 
e^{-r_\ell |x-z|}\|q_1 - q_2\|_{C}  e^{\frac12r_\ell |x-z|}           e^{\frac14 r_\ell |x-z|} 
  \le 
e^{-\frac14r_\ell |x-z|}\|q_1 - q_2\|_{C} .
 \end{align*}
Hence, for all $x \in \rr$, we obtain 
\[
\left|\int_\rr J (x,z) dz \right| (x) \le \|q_1 - q_2\|_{C}  \int_\rr e^{-\frac14r_\ell |x-z|} dz  = \frac{8}{r_\ell}\|q_1 - q_2\|_{C}   , 
\]
for all $x$. 

Inserting the above estimates back into \eqref{p2-est}, we obtain
\begin{align}
\|P_2(w_1,q_1) - P_2(w_2,q_2)\|_C &\leq \frac{2r}{r_l}\|q_1-q_2\|_C+\frac{2r_u}{r_l}\|w_1-w_2\|_{C}+\frac{8rr_u}{r_l}\|q_1+q_2\|_C \nonumber \\
& \leq \frac{8rr_u}{r_l}\|y_1-y_2\|_Y \nonumber \\
& \leq 10r\|y_1-y_2\|_Y.
\end{align}
Factoring in the $\frac12$ from the integral, we may sharpen this to
\begin{equation}
\|P_2(w_1,q_1) - P_2(w_2,q_2)\|_C\leq 5r\|y_1-y_2\|_Y.
\end{equation}

\subsection{Estimating $E_1$}
We now proceed in estimating $E_1$.  We find that
\begin{align}
\label{pest1-2}
P_1(w_1,q_1) - P_1(w_2,q_2) &= \frac12 \int_{x}^{\infty}e^{-|\int_{x}^{z}q_1(y,t)dy|}w_1(z,t)q_1(z,t)dz \nonumber \\ 
&- \frac12 \int_{x}^{\infty}e^{-|\int_{x}^{z}q_2(y,t)dy|}w_2(z,t)q_2(z,t)dz \nonumber \\ 
&- \frac12\int_{-\infty}^{x}e^{-|\int_{x}^{z}q_1(y,t)dy|}w_1(z,t)q_1(z,t)dz \nonumber \\ 
&+ \frac12\int_{-\infty}^{x}e^{-|\int_{x}^{z}q_2(y,t)dy|}w_2(z,t)q_2(z,t)dz.
\end{align}
We will estimate the difference in the first two integrals in \eqref{pest1-2} in the $C^1$ norm.  The estimates for the other two are similar.  Set
\begin{align}
I_1 &= \frac12 \int_{x}^{\infty}e^{-|\int_{x}^{z}q_1(y,t)dy|}w_1(z,t)q_1(z,t)dz \nonumber \\ 
&- \frac12 \int_{x}^{\infty}e^{-|\int_{x}^{z}q_2(y,t)dy|}w_2(z,t)q_2(z,t)dz.
\end{align}
By triangle inequality, we have that
\begin{align}
\label{i-ests}
\|I_1\|_{C^1} &\leq \left\|\frac12 \int_{x}^{\infty}e^{-|\int_{x}^{z}q_1(y,t)dy|}w_1(z,t)q_1(z,t)dz-\frac12 \int_{x}^{\infty}e^{-|\int_{x}^{z}q_1(y,t)dy|}w_1(z,t)q_2(z,t)dz\right\|_{C_1} \nonumber \\ 
&+\left\|\frac12 \int_{x}^{\infty}e^{-|\int_{x}^{z}q_1(y,t)dy|}w_1(z,t)q_2(z,t)dz-\frac12 \int_{x}^{\infty}e^{-|\int_{x}^{z}q_1(y,t)dy|}w_2(z,t)q_2(z,t)dz\right\|_{C_1} \nonumber \\ 
&+\left\|\frac12 \int_{x}^{\infty}e^{-|\int_{x}^{z}q_1(y,t)dy|}w_2(z,t)q_2(z,t)dz-\frac12 \int_{x}^{\infty}e^{-|\int_{x}^{z}q_2(y,t)dy|}w_2(z,t)q_2(z,t)dz\right\|_{C_1} \nonumber \\ 
& = \|I_{1,1}\|_{C^1}+\|I_{1,2}\|_{C^1}+\|I_{1,3}\|_{C^1}.
\end{align}
\subsubsection{Estimating $I_{1,1}$}  We must estimate both the sup norm of the difference in integrals and its derivative.  We find by previous estimates that 
\begin{equation}
\label{supi11}
\|I_{1,1}\|_{L^\infty} \leq r\|q_1 - q_2\|_{L^\infty}\int_x^\infty e^{-r_l|x-z|}dz = \frac{r}{r_l}\|q_1 - q_2\|_{L^\infty}.
\end{equation}
We then have by the Fundamental Theorem of Calculus and chain rule that
\begin{align}
\label{supci11}
|\p_xI_{1,1}| &\leq \int_x^\infty\left|e^{-|\int_{x}^{z}q_1(y,t)dy|}q_1(x)w_1(z)(q_1-q_2)(z)\right|dz \nonumber \\ 
& +\|w_1q_1-w_1q_2\|_{L^\infty} \nonumber \\ 
& \leq \left(\frac{rr_u}{r_l}+r\right)\|q_1 - q_2\|_{L^\infty}.
\end{align}
\subsubsection{Estimating $I_{1,2}$}  Using similar techniques to estimates for $I_{1,1}$, we find that
\begin{equation}
\label{supi12}
\|I_{1,2}\|_{L^\infty} \leq r_u\|w_1 - w_2\|_{L^\infty}\int_x^\infty e^{-r_l|x-z|}dz = \frac{r_u}{r_l}\|w_1 - w_2\|_{L^\infty}.
\end{equation}
Again, by the Fundamental Theorem of Calculus and chain rule we have that
\begin{align}
\label{supci12}
|\p_xI_{1,2}| &\leq \int_x^\infty\left|e^{-|\int_{x}^{z}q_1(y,t)dy|}q_1(x)q_2(z)(w_1-w_2)(z)\right|dz \nonumber \\ 
& +\|w_1q_2-w_2q_2\|_{L^\infty} \nonumber \\ 
& \leq \left(\frac{r_u^2}{r_l}+r_u\right)\|w_1 - w_2\|_{L^\infty}.
\end{align}
\subsubsection{Estimating $I_{1,3}$}  Using similar techniques to those found in the estimates for $E_2$, we may conclude that
\begin{equation}
\label{i13}
\|I_{1,3}\|_{C^1} \leq \frac{8rr_u^2}{r_l}\|q_1 - q_2\|_{L^\infty}.
\end{equation}
By collecting estimates on $I_{1,1}$ - $I_{1,3}$, we obtain
\begin{equation}
\label{i1-1}
\|I_1\|_{C^1} \leq \frac{8rr_u^2}{r_l}\|y_1 - y_2\|_{Y} \leq \frac{100}{9}r\|y_1 - y_2\|_{Y}.
\end{equation}
Factoring in the $\frac12$ in front of our integrals, we may sharpen this estimate to 
\begin{equation}
\label{i1-2}
\|I_1\|_{C^1}\leq \frac{50}{9}r\|y_1 - y_2\|_{Y}.
\end{equation}
The last two integrals under the $C^1$ norm in the difference between $P_1(w_1,q_1) - P_1(w_2,q_2)$ are estimated similarly.  Collecting all of our estimates on $E_1$ - $E_3$ concludes our proof.
\end{proof}
After applying the ODE theorem, we have that there exists a unique solution vector 
$$
\left(
\begin{array}{c}
w\\
v\\
q
\end{array}
\right)  \in C^1 ( [-T, T ] ; Y ), \ \ \ T=\frac{1}{2L} = \frac{9}{100r},
$$
where $L$ was given in the previous lemma;
i.e we have  solutions 
\begin{align} 
w \in C^1 ( [-T, T ] ; C^1 ), \quad v\in  C^1 ( [-T, T] ; C), \quad q\in  C^1 ( [-T, T ] ; C ).
\end{align}
In the next section, we will construct a diffeomorphism, labeled $\eta(x,t)$, using $w(x,t)$. Then, we can construct the solution $u(x,t)$ to the FW initial value problem.

%%%%%%%%%%%%%%%%%%%%%%%%%%%
%
%
%    Construction of u(x,t)
%
%
%%%%%%%%%%%%%%%%%%%%%%%%%%%
\section{Construction of $u(x,t)$}
Let  $r$   be the  constant such that  $ \sup_{t\in [-T,T] } \| w(t)\| _{C^1} \le r$.  Next we define the function 
\begin{align}
\eta(x,t) \dot = x + \frac32\int_0^t w (x,\tau) d\tau.
\end{align}
We shall show that for $t \in [-T, T]$, $\eta(\cdot,t)$ is differentiable and $\eta_x >0$, which implies $\eta \in \text{Diff}^1$. 
\begin{lemma} For all $t \in [-T, T]$ fixed,  $\eta(\cdot,t) \in \text{Diff}^1$ and  $\p_t \eta(x,t) \in C^1(  [-T, T ] ; C^1 )$. 
\end{lemma} 
\begin{proof}
First we observe that since   $w(x, \cdot) \in C^1$, $\eta(x, \cdot) -x \in C^1$. Thus, we may differentiate $\eta$ with respect to $x$ to obtain 
$$
\p_x \eta(x,t) = 1 + \frac32\int_0^t \p_x w(x,\tau ) d\tau. 
$$
Since $w  \in C^1(  [-T, T] ; C^1 )$ is in the ball of radius $r$, we know $ |  w(x,\tau ) |,| \p_x w(x,\tau ) |\le r$. Hence
$$
1 -\frac{3}{2}rt \le \partial_x\eta(x,t) \le 1 + \frac{3}{2}rt.
$$
Thus, for all $t \in [-T,T]$ such that $t < \frac9{ 100r}$, $\p_x\eta(x,t) \ge \frac{173}{200}>0$.
Next we consider $\eta^{-1} (x,t)$. It satisfies 
$$
\eta^{-1} (x,t) = x -\frac32\int_0^t w( \eta^{-1} (x,\tau)  ,\tau ) d\tau 
$$ 
By the inverse function theorem
$$
\p_x \eta^{-1} (x,t) =1  -\frac32\int_0^t  \p_x w( \eta^{-1} (x,\tau)  ,\tau )  \frac{1} { \p_x \eta(\eta^{-1} (x,\tau) ,\tau) } d\tau .
$$ 
Thus, we obtain 
$$
 1 - \frac{3rt}{2+3rt}  \le \p_x \eta^{-1} (x,t) \le 1 + \frac{3rt}{2-3rt}  ,
$$ 
using $t<\frac9{100r}$ we have 
$$
 \frac{200}{227} = 1 - \frac{\frac{27}{100} }{2+\frac{27}{100} }  \le \p_x \eta^{-1} (x,t) \le 1 + \frac{\frac{27}{100} }{2-\frac{27}{100} }  = \frac{200}{173} .
$$ 
Hence, $\eta(x, \cdot) \in \text{Diff}^1$. For the second claim, we differentiate $\eta$ with respect to $t$ and obtain 
$$
\p_t \eta(x,t) =  \frac32w(x,t) \in C^1(  [-T, T ] ; C^1 ),
$$
which completes the claim. 
\end{proof} 
Using $\eta^{-1}$ we define the function
 $$
u(x,t) \dot = w ( \eta^{-1} (x,t) , t) 
.
$$
By construction, $u$ satisfies the FW equation.
In fact, we have from our construction that
\begin{align*}
\p_tu(x,t)&=\p_t[w(\eta^{-1}(x,t),t)] \\ 
&=\p_tw(\eta^{-1}(x,t),t)+\p_xw(\eta^{-1}(x,t),t)\p_t\eta^{-1}(x,t).
\end{align*}
Since we know that
\begin{align*}
\p_tw(\eta^{-1}(x,t),t)&=-\p_x(1-\p_x^2)^{-1}(u)
\end{align*}
and
\begin{align*}
&\p_xw(\eta^{-1}(x,t),t)=\p_xu(x,t) \\ 
&\p_t\eta^{-1}(x,t)=-\frac32\p_t\int_0^t w( \eta^{-1} (x,\tau)  ,\tau ) d\tau = -\frac32w(\eta^{-1}(x,t),t)=-\frac32u(x,t),
\end{align*}
we have that
\begin{align*}
\p_tu(x,t) & =-\p_x(1-\p_x^2)^{-1}(u)-\frac32u\p_xu
\end{align*}
Thus, $u(x,t)$ satisfies the FW equation.
It remains to check that $u \in C([-T, T]; C^1) \cap C^1 ( [-T, T]; C)$. 

\begin{lemma}
$u \in C([-T, T]; C^1) \cap C^1 ( [-T, T]; C)$. 
\end{lemma}
\begin{proof}
First we check that $u \in C([-T, T]; C^1)$. 
Notice, 
$t \mapsto u(x,t)$ is bounded and continuous, since $w ( \eta^{-1} (x,t) , t) $ is bounded and continuous in $t$, in particular, 
$$
\sup_{x\in \rr}  | w ( \eta^{-1} (x,t) , t) | = \sup_{x\in\rr} | w ( x  , t) | \le r. 
$$
Next we differentiate $u$ with respect to $x$, and we find 
$$
\p_x u(x,t) = w_x(\eta^{-1}(x,t) , t) \eta^{-1}_x(x,t).
$$
The functions on the right hand side are continuous in $t$. 
Using the inverse function theorem, we obtain 
$$
|\p_x u(x,t) |= |w_x(\eta^{-1}(x,t) , t)| \frac{1}{\eta_x(\eta^{-1}(x,t),t)} .
$$
Using $\frac{173}{200}  \le \p_x \eta(x,t) $, and $w\in C^1(  [-T, T] ; C^1 )$ is in the ball of radius $r$, we obtain
$$
|\p_x u(x,t) |\le\frac{200}{173} r < 2r .
$$
Taking the supremum over $t \in [-T, T]$, we obtain $u \in C([-T, T]; C^1) $ as well as the solution size estimate 
\begin{align}
\sup_{t\in [-T,T]} \| u(t) \|_{C^1} \le2 \| u_0\|_{C^1} . 
\end{align}

Next we show $u \in C^1 ( [-T, T]; C)$. To do so, we differentiate $u$ with respect to $t$ to obtain 
$$
u_t (x,t) = w_x(\eta^{-1}(x,t) , t) \eta^{-1}_t(x,t) + w_t(\eta^{-1}(x,t) , t) .
$$
We have 
$$
 w_t(\eta^{-1}(x,t) , t) \in C([-T, T]; C^1),  \quad\eta^{-1}_t(x,t)  = -\frac32w (\eta^{-1} (x,t),t) \in C([-T, T]; C^1), 
$$
 and  finally
$$
w_x(\eta^{-1}(x,t) , t)  \in C( [-T, T] ; C).
$$
Since the last term is only continuous in the spacial variable, we have $u  \in C^1 ( [-T, T]; C)$. 
\end{proof}

The next step in our proof is to show that $u$ is the unique solution in these spaces. 
\begin{lemma}
There is only one solution $u \in C([-T, T]; C^1)$ to the Cauchy problem for the FW equation which can be constructed as  above. 
\end{lemma}
\begin{proof}
It follows from the ODE theorem, that given initial data $u_0$, there exists a unique solution $w$ to the first equation in the ODE system. 
Given $w$, $\eta-x$ is uniquely defined, since 
$$
\eta - x = \frac32\int_0^t w d\tau \in C^1([-T,T]; C^1).
$$
By uniqueness of inverses and composition, it follows immediately that $u = w \circ \eta^{-1} $ is uniquely defined. 
\end{proof}
The final step in our proof, is  to show the data--to--solution map is continuous. 
\begin{lemma} 
The data--to--solution map $u_0\mapsto u$ from $C^1$ to $C([-T, T]; C^1)$, is continuous. 
\end{lemma} 
\begin{proof}
Let $\{ u_{\ee,0} \}_{\ee \in (0,1]} $ be a sequence of $C^1$ functions in the ball of radius $r$ where $\lim _{\ee \rightarrow 0} u_{\ee,0} = u_0$, and let $u_\ee$ and $ u$ be the corresponding solutions constructed above. 
For every $t$ fixed, consider
$$
\lim_{\ee\rightarrow 0} \| u_\ee(t) - u(t)\|_{C^1} =\lim_{\ee\rightarrow 0} \| w_\ee \circ \eta_\ee^{-1} (t)- w\circ \eta^{-1} (t)\|_{C^1}. 
$$
We have from the ODE theorem, that $w_\ee$ and $\p_x w_\ee$ depend continuously on $w_{\ee,0}$ (in fact, the dependence is Lipshitz). 
Thus $\lim_{\ee \rightarrow 0} w_\ee = w$ in $C([-T,T];C^1)$. Since the composition and product of continuous functions is continuous, it suffices to show therefore that for every $x$ and $t$ in the domain, $\eta_\ee^{-1} (x,t)$ and $\p_x \eta_\ee^{-1} (x,t)$ depend continuously on $\ee$. That is, the pointwise convergence
$$
\lim_{\ee \rightarrow 0} \eta_\ee^{-1} (x,t)= \eta^{-1} (x,t)\quad \text{ and  } \quad \lim_{\ee \rightarrow 0}\p_x  \eta_\ee^{-1} (x,t)= \p_x \eta^{-1} (x,t). 
$$
We begin by showing the first limit. 
For each $t$ in the interval of existence, let $y_\ee$ be an arbitrary number, and find $x$ such that $ y_\ee =\eta_\ee(x,t)$. Given the aforementioned $x$, let $y  = \eta (x,t)$, then we have the following   equalities
\begin{align}
& \eta^{-1} (y_\ee,t) - \eta^{-1} _\ee (y_\ee,t)  =  \eta^{-1} ( y_\ee-y  +y ,t)-x  = \eta^{-1}( y_\ee-y  +y,t) -\eta^{-1}(y,t)  .
\end{align}
Thus we have 
\begin{align}
&\lim_{\ee \rightarrow 0} \left(  \eta^{-1} (y_\ee,t) - \eta^{-1} _\ee (y_\ee,t)  \right) = \lim_{\ee \rightarrow 0}  \eta^{-1}( y_\ee-y  +y,t) -\eta^{-1}(y,t)  .
\end{align}
Since $\eta^{-1} (t)\in C^1$, we can push the limit inside, from which we obtain 
\begin{align}
&\lim_{\ee \rightarrow 0} \left(  \eta^{-1} (y_\ee,t) - \eta^{-1} _\ee (y_\ee,t)  \right) =\eta^{-1}(  y +  \lim_{\ee \rightarrow 0}  (y_\ee-y) ,t) -\eta^{-1}(y,t)  .
\end{align}
Thus if $ \lim_{\ee \rightarrow 0}  (y_\ee-y) = 0$, then we may conclude 
\begin{align}
&\lim_{\ee \rightarrow 0} \left(  \eta^{-1} (y_\ee,t) - \eta^{-1} _\ee (y_\ee,t)  \right) =0,
\end{align}
at every $y_\ee \in \rr$. This gives us the pointwise limit we desire. 
To show $ \lim_{\ee \rightarrow 0}  (y_\ee-y) =0$, we calculate 
$$
 \lim_{\ee \rightarrow 0}  (y_\ee-y) = \lim_{\ee \rightarrow 0}   (\eta_\ee(x,t) - \eta(x,t) ) =  \lim_{\ee \rightarrow 0} \frac32\int_0^t  w_\ee(x,\tau) - w(x,\tau) d\tau = 0, 
$$
since $w_\ee \rightarrow w$ in $C([-T,T]; C^1)$. 
Next, we shall show that $\lim_{\ee \rightarrow 0}\p_x  \eta_\ee (x,t)= \p_x \eta (x,t)$. Indeed, we have 
\begin{align}
\p_x \eta_\ee (x,t) = 1 + \frac32\int_0^t \p_x w_\ee (x,\tau) d\tau . 
\end{align}
Since $w_\ee \rightarrow w \in C([-T,T]; C^1) $ as $\ee \rightarrow 0$, we conclude 
\begin{align}
 \lim_{\ee \rightarrow 0}  \p_x \eta_\ee (x,t) = 1 + \frac32\int_0^t \p_x w (x,\tau) d\tau =  \p_x \eta (x,t) . 
\end{align}
We now show $\lim_{\ee \rightarrow 0}\p_x  \eta_\ee^{-1} (x,t)= \p_x \eta^{-1} (x,t)$. 
We have  by the inverse function theorem
$$
 \p_x  \eta_\ee^{-1} (x,t)=\frac{1}{\p_x \eta_\ee (\eta_\ee^{-1}(x,t) ,t)}.  
$$
Since $\p_x \eta_\ee(x,t)$ is continuous in the spacial variable, and both $\p_x \eta_\ee (x,t) \rightarrow \p_x \eta(x,t)$ and $\eta_\ee^{-1} (x,t) \rightarrow \eta^{-1} (x,t) $, we may conclude 
\begin{align}
 \lim_{\ee \rightarrow 0}  \p_x  \eta_\ee^{-1} (x,t)=\frac{1}{   \p_x \eta (\eta^{-1}(x,t),t)}=\p_x \eta ^{-1}(x,t).  
\end{align}
In other words, we have established the claimed convergence.  This concludes our proof of Theorem \ref{wp}.
\end{proof}
\section{H\"older continuity of the data--to--solution map}
The data--to--solution map is H\"older continuous if we consider a weaker topology. The next lemma shows that the data--to--solution map is Lipschitz continuous if we consider the $C^0$ topology. 
\begin{lemma}
The data--to--solution map $u_0\mapsto u$ from $C^0$ to $C([-T, T]; C^0)$, is Lipschitz. 
\end{lemma}
\begin{proof}
Let $\{ u_{\ee,0} \}_{\ee \in (0,1]} $ be a sequence of $C^1$ functions in the ball of radius $r$, and let $u_\ee$ be the corresponding solutions constructed in the previous section. 
We will show that 
$$
\sup_{x\in \rr}  |u_{\ee}(x,t) - u  (x,t) | \le c \sup_{x\in \rr}  |u_{\ee,0} (x) - u_0 (x)  | .
$$
Let us define 
$$
y_{\ee,0} = \left(
\begin{array}{c}
u_{\ee,0}\\
\frac{d}{dx} u_{\ee,0}\\
1
\end{array}
\right) ,
$$
and let $y_\ee  = (w_\ee, v_\ee, q_\ee) \in C^1( [-T, T; Y) $ be the corresponding solution. 
Then the mapping $y_{\ee,0}  \mapsto y_\ee$ is locally Lipschitz continuous, by the ODE theorem. In particular, for all $t$ sufficiently small, there exists a constant $c$, independent of $\ee$ such that
$$
\| w_\ee (t) - w (t) \|_{C^1} \le c\|w_{\ee,0} - w_0\|_{C^1} =c  \|u_{\ee,0} - u_0\|_{C^1} .
$$
Using the above functions, we have 
$$
\eta_\ee (x,t) - \eta(x,t)  = \frac32\int_0^t [w_\ee(x,\tau)  - w (x,\tau)] d\tau, 
$$
from which we can conclude for $s = 0,1$
$$
\|  \eta_\ee (x,t) - \eta(x,t)  \|_{C^s} \le  \frac32\int_0^t \|  w_\ee(x,\tau)  - w (x,\tau) \|_{C^s}  d\tau \le c t \|u_{\ee,0} - u_0\|_{C^s}.
$$
Using the Lipschitz continuity of $\eta$, we will show that the corresponding inverse functions converge
\begin{align}
\lim_
{\ee\rightarrow 0} \eta^{-1}_\ee(x,t) = \eta^{-1} (x,t). 
\end{align}
For each $t$ in the interval of existence, let $y$ be an arbitrary number such that $ y =\eta(x,t)$, for some $x\in \rr$ and for $\ee>0$, let $y_\ee = \eta_\ee(x,t)$, then
$$
| \eta^{-1} (y,t) - \eta^{-1} _\ee (y,t) | = |x - \eta^{-1}_\ee( y-y_\ee +y_\ee,t) | = |\eta^{-1}_\ee(y_\ee,t) - \eta^{-1}_\ee( y-y_\ee +y_\ee,t) |.
$$
Since 
$\eta_\ee (x,t) \in C([-T,T];C^1)$, we have 
$$
| \eta^{-1} (y,t) - \eta^{-1} _\ee (y,t) | \le \sup_{x \in \rr}  |\p_x \eta^{-1}_\ee (x,t) | |y_\ee - ( y-y_\ee +y_\ee) | = 
 \sup_{x \in \rr} \left |\frac{1}{\p_x \eta_\ee (\eta_\ee^{-1}(x,t) ,t)} \right| |  y-y_\ee  | .
$$
Substituting the definitions of $y$ and $y_\ee$, and using $\frac{173}{200} \le \p_x \eta_\ee (\eta_\ee^{-1}(x,t) ,t)   $ we obtain 
\begin{align}\label{eta_inv_est} 
| \eta^{-1} (y,t) - \eta^{-1} _\ee (y,t) | \le  
\frac{200}{173} | \eta(x,t)-\eta_\ee(x,t)  | \le 2 c t \|u_{\ee,0} - u_0\|_{C^s} ,
\end{align}
for $s = 0,1$.
Thus $ \lim_{\ee \rightarrow 0} \eta^{-1} _\ee (x,t) =\eta^{-1} _\ee (x,t) \in C([-T,T]; C^0) .$
Finally, we consider the difference $|u_{\ee}(x,t) - u  (x,t) |$. 
We have, 
$$
|u_{\ee}(x,t) - u  (x,t) | = |w_{\ee}(\eta_\ee^{-1} (x,t),t) - w (\eta^{-1} (x,t),t) |.
$$
By the triangle inequality, this is bounded by 
$$
|u_{\ee}(x,t) - u  (x,t) | \le  |w_{\ee}(\eta_\ee^{-1} (x,t),t) -w_{\ee}(\eta^{-1} (x,t),t)  |+
| w_{\ee}(\eta^{-1} (x,t),t)  - w (\eta^{-1} (x,t),t) |.
$$
Since $\eta(x,t)$ is a diffeomorphism, the second term is bounded by 
$$
| w_{\ee}(\eta^{-1} (x,t),t)  - w (\eta^{-1} (x,t),t) | \leq
 \sup_{x\in \rr} | w_{\ee}(x,t)  - w (x,t) |,
$$
which is bounded by the $C^s$, $s=0,1$,  norm.  Thus we have
$$
|u_{\ee}(x,t) - u  (x,t) | \le  |w_{\ee}(\eta_\ee^{-1} (x,t),t) -w_{\ee}(\eta^{-1} (x,t),t)  |+
\| w_{\ee}(t)  - w (t) \|_{C^s}.
$$
which, by previous estimation of the above second term, gives us 
$$
|u_{\ee}(x,t) - u  (x,t) | \le  |w_{\ee}(\eta_\ee^{-1} (x,t),t) -w_{\ee}(\eta^{-1} (x,t),t)  |+
c\| u_{\ee,0} - u_0\|_{C^s} .
$$
We apply the mean value theorem to the first term, from  which we conclude
$$
|u_{\ee}(x,t) - u  (x,t) | \le   \sup_{x\in \rr} |\p_x w_{\ee} (x,t) ||\eta_\ee^{-1} (x,t)-\eta^{-1} (x,t)   |+
c\| u_{\ee,0} - u_0\|_{C^s} .
$$
Using $w_\ee \in C([-T,T];C^1)$ and is bounded by $r$, and inequality \ref{eta_inv_est}, we obtain 
$$
|u_{\ee}(x,t) - u  (x,t) | \le   2r c t \|u_{\ee,0} - u_0\|_{C^s} +
c\| u_{\ee,0} - u_0\|_{C^s} ,
$$
which shows, in particular, that the data--to--solution map is Lipschitz continuous from $C^0 $ to $C([-T,T];C^0) $. 
\end{proof}

Now that we have established continuity of the data--to--solution map in the $C^1$ topology and Lipschitz continuity in the $C^0$ topology, we may conclude that the continuity is H\"older continuous in the intermediate topologies.
\begin{lemma}
$
\sup_{t \in [-T, T]} \| u_\ee (t)- u(t)\|_{C^{0,\alpha} } \lesssim \| u_{\ee,0} - u_0\|_{C^{0,\alpha} } ^{1-\alpha} . $
\end{lemma}
\begin{proof}
Let $v(x,t) =  u_\ee(x,t)  - u(x,t) $. Then by definition 
\begin{align}
\| v(t)\|_{C^{0,\alpha} }  = \sup_{x\neq y} \frac{ |v(x,t) -v( y,t)|}{|x-y|^\alpha}  = \sup_{x\neq y}  |v(x,t) -v( y,t)|^{1-\alpha}  \frac{ |v(x,t) -v( y,t)|^{\alpha} }{|x-y|^\alpha}  .
\end{align}
This is bounded by  
\begin{align}
\| v(t)\|_{C^{0,\alpha} }  \le (2 \| v(t)\|_{C^0})^{1-\alpha}  \| v(t)\|_{C^1} ^\alpha  .
\end{align}
Using the Lipschitz continuity in $C^0$, and the fact that $ \| v(t)\|_{C^1}  \le 2 r$, we obtain 
\begin{align}
\| v(t)\|_{C^{0,\alpha} }  \le   2 r^\alpha \| v(0)\|_{C^0}^{1-\alpha}   ,
\end{align}
which completes the proof of the lemma.
\end{proof}

\end{document}